\documentclass[11pt,a4paper]{article}
\usepackage{amsmath}
\usepackage{amssymb}
\usepackage{amsthm}

\usepackage{cite}
\title{Majority-logic Decoding with Subspace Designs}
\author{Romar dela Cruz\thanks{Institute of Mathematics, University of the Philippines Diliman, Quezon City, Philippines,
currently at Mathematisches Institut,
University of Bayreuth, D-95440 Bayreuth, Germany.
\texttt{email:~rbdelacruz@math.upd.edu.ph}.
The work of R. dela Cruz is supported by the Alexander von Humboldt Foundation.}
\and
Alfred Wassermann\thanks{Mathematisches Institut,
University of Bayreuth, D-95440 Bayreuth, Germany.
\texttt{email:~alfred.wassermann@uni-bayreuth.de}. This paper was presented in part at the Oberwolfach Workshop 1912 ``Contemporary Coding Theory''.}}

\newcommand{\F}{\mathbb{F}}
\newcommand{\cB}{\mathcal{B}}
\newcommand{\cD}{\mathcal{D}}
\newcommand{\cA}{\mathcal{A}}
\newcommand{\cF}{\mathcal{F}}
\newcommand{\cG}{\mathcal{G}}
\newcommand{\cP}{\mathcal{P}}
\newcommand{\cR}{\mathcal{R}}
\newcommand{\qbinom}[3]{\genfrac{[}{]}{0pt}{}{#1}{#2}_{#3}}

\DeclareMathOperator{\rank}{rank}
\DeclareMathOperator{\AG}{AG}
\DeclareMathOperator{\PG}{PG}

\newtheorem{defini}{Definition} %[section]
\newtheorem{thm}[defini]{Theorem}
\newtheorem{rem}[defini]{Remark}

\begin{document}
\maketitle

\begin{abstract}
Rudolph (1967) introduced one-step majority logic decoding for linear codes
derived from combinatorial designs.
The decoder is easily realizable in hardware and requires that
the dual code has to contain the blocks of
so called geometric designs as codewords.
Peterson and Weldon (1972) extended Rudolphs algorithm to a two-step
majority logic decoder correcting the same number of errors than
Reed's celebrated multi-step majority logic decoder.

Here, we study the codes from subspace designs.
It turns out that these codes have the same majority logic decoding capability
as the codes from geometric designs, but their majority logic decoding complexity
is sometimes drastically improved.
\end{abstract}
% \begin{IEEEkeywords}
% Majority logic decoding, error-correcting codes, combinatorics, designs
% \end{IEEEkeywords}

\section{Introduction}
In \cite{Rudolph1967}, a simple decoding method based on majority decision for linear codes is presented.
Its attraction lies in the easy realization in hardware and it requires that
the dual code has to contain the blocks of a $t$-design, $t\ge 2$, as codewords.

Ever since then, people studied the linear codes generated by the blocks of $t$-designs.
In order to get a good code it is desirable that the rank of the block-point incidence
matrix of the design is small over some finite field.
The famous Hamada conjecture states that \emph{geometric designs}, which consist
of the set of all $k$-subspaces in PG$(v,q)$ or $k$-flats in AG$(v,q)$, minimize the $p$-rank for
a prime power $q=p^s$.

Here, a simple observation on the codes from subspace designs---also known as
$q$-analogs of designs---is reported.
It will turn out that these codes have the same majority logic decoding capability
as the codes from geometric designs, but their decoding complexity is improved.

This may be of interest when implementing error correction with
nano-scale technologies~\cite{ldpc2013}.

% % % % % % % % % % % % % % % % % % % % % % % % % % % % % % % % % % % % % % % % % % % % % % %
\section{Combinatorial designs}
For a finite set $V$ of cardinality $v$, the notion of a design goes back to
Pl\"ucker, Kirkman and Steiner in the 19th century.

\begin{defini}\label{def:design}
Let $0\leq t\leq k\leq v$ be integers and $\lambda$ a non-negative integer.
A pair $\cD=(V,\cB)$, where $\cB$ is a collection of subsets of cardinality $k$ (\emph{blocks}) of $V$,
is called a \emph{$t$-$(v,k,\lambda)$ design} on $V$
if each subset of cardinality $t$ of $V$ is contained in exactly $\lambda$ blocks.

If $\cB$ is a set, i.e.~if every $k$-subset appears at most once in $\cB$, the design is called \emph{simple}.
\end{defini}

It is well known, see e.g. \cite[1\S3, Thm. 3.2]{beth1999design}, that every $t$-$(v,k,\lambda)$ design is
also an
$s$-$(v, k, \lambda_s)$ design for $0\leq s \leq t$, where
\[
\lambda_s = \lambda\frac{\binom{v-s}{t-s}}{\binom{k-s}{t-s}}\,.
\]
As a consequence,
a $t$-$(v,k,\lambda)$ design consists of $b=\lambda \binom{v}{t} / \binom{k}{t}$ blocks and every point
$P\in V$ appears in $r = \lambda \binom{v - 1}{t - 1} / \binom{k - 1}{t - 1}$ blocks.
$r$ is called \emph{repetition number}.

%\subsection{Codes with majority logic decoding}
Rudolph \cite{Rudolph1967} suggested to use the
rows of a $b\times v$ blocks-points-incidence matrix $N_\cD$ of a $2$-$(v,k,\lambda)$  design $\cD$
as parity check equations for a linear code $C_\cD$ over $\F_p$.
In other words, the rows of $N_\cD$ span the dual code
$C_\cD^\perp$ of length $v$ over $\F_p$.
In~\cite{Rudolph1967,Ng1970} it is shown that
with \emph{one-step majority logic decoding}
the number of errors which can be decoded is equal to
$
	\lfloor (r + \lambda - 1) / 2\lambda\rfloor
$.
We note that for each coordinate of a received word the decoder uses those
$r$ parity check equations that contain that coordinate plus one additional equation.
Thus, the complexity of the decoder is dominated by the repetition number $r$ of the design.
\cite{RahmanBlake1973} extended the analysis of the majority logic decoder to designs with arbitrary $t\ge 2$,
see \cite[p. 686]{Tonchev-2007} for a survey.

Let $R$ be the rank of $N_\cD$ over $\F_p$ (the \emph{$p$-rank of $\cD$}).
Then $C_\cD^\perp$ is a linear $[v, R]_p$ code and $C_\cD$ is a $[v, v-R]_p$ code.
This suggests that it is interesting to search for designs $\cD$ with small $p$-rank.
The following theorem by Hamada shows that only the codes of designs with
a special restriction on the parameters may be interesting.

\begin{thm}[\cite{hamada1973}]
Let $N$ be the incidence matrix of a $2$-design with parameters
$v$, $k$, $\lambda$, $r$, $b$, and let $p$ be a prime.
\begin{itemize}
\item If $p$ does not divide $r(r-\lambda)$, then $\rank_p N = v$.
\item If $p$ divides $r$ but does not divide $r-\lambda$, then $\rank_p N \ge v-1$.
\item If $\rank_p N < v-1$, then $p$ divides $r-\lambda$.
\end{itemize}
\end{thm}

\section{Geometric designs and their codes}
For certain designs derived from finite geometry it is known that
their $p$-rank is smaller than $v-1$.
These designs are the so called \emph{geometric or classical designs}, \cite{hamada1968, beth1999design}.

Let $q$ be a prime power $p^m$ and $V$ be a vector space of finite dimension $v$ over the finite field $\F_q$.
For $0\leq k \leq v$, we denote the set of $k$-dimensional subspaces of $V$ with
$
\qbinom{V}{k}{q} = \{ U \le V\mid \dim{U} = k\}
$.
The cardinality of $\qbinom{V}{k}{q}$ can be expressed by the
\emph{Gaussian coefficients}
$
	\#\qbinom{V}{k}{q} = \qbinom{v}{k}{q} = \frac{(q^v-1)\cdots(q^{v-k+1}-1)}{(q^k-1)\cdots(q-1)}\text{.}
$

\subsection{Projective case}
Taking
$\cP = \qbinom{V}{1}{q}$ as set of points
and $\cB = \qbinom{V}{k}{q}$ as set of blocks,
then it is well known \cite[1\S2]{beth1999design} that $\cG = (\cP, \cB)$ is a
\[
	2\text{-}(\qbinom{v}{1}{q}, \qbinom{k}{1}{q}, \qbinom{v - 2}{k - 2}{q})
\]
design. $\cG$ is called \emph{geometric or classical design}.
We note that $r= \qbinom{v-1}{k-1}{q}$ and $b= \qbinom{v}{k}{q}$.
In the language of finite geometry, the points of the geometric design are the points
of $\PG(v-1, q)$ and the blocks are the $(k-1)$-subspaces of $\PG(v-1, q)$.

The $p$-rank for a geometric design $\cG$ has been determined by Hamada ~\cite{hamada1968}:
\begin{equation}
\rank_p N_\cG = \sum_{s_0,\ldots,s_{m-1}}
	\prod_{j=0}^{m-1} \sum_{i=0}^{L(s_{j+1},s_j)}
	(-1)^i \binom{v}{i}\binom{v-1+s_{j+1}p-s_j-ip}{v-1},
\label{eq:hamada}
\end{equation}
where $s_m = s_0$,
$k \leq s_j \leq  v$, $0\leq s_{j+1}p-s_j \leq v(p-1)$, and
$L(s_{j+1},s_j) = \lfloor (s_{j+1}p-s_j)/p\rfloor$.

The code $C_\cG$ is called \emph{Projective Geometry code} (PG code),
see e.g. \cite{PetersonWeldon1972}.
It's minimum distance is at least $\qbinom{v-k+1}{1}{q}$ \cite[Thm. 5.7.9]{assmus_key_1992}.

\subsection{Affine case}
Similarly, the points  and the $k-1$-flats of the affine geometry $\AG(v-1, q)$ form a
\[
   2\text{-}(q^{v-1}, q^{k-1}, \qbinom{v-2}{k-2}{q})
\]
design $\cA$. It is called \emph{geometric design}, too.
The rank of $\cA$ is related to that of the projective case, see~\cite{hamada1968}.

The code $C_\cA$ is known as \emph{Euclidean Geometry code} (EG code),
see \cite{PetersonWeldon1972}.
It's minimum distance is at least $2q^{v-k}$ \cite[Thm. 5.7.9]{assmus_key_1992}.

% % % % % % % % % % % % % % % % % % % % % % % % % % % % % % % % % % % % % % % % % % % % % % %
\subsection{The binary case}
It is well known that for $p=q=2$
$
    C_\cA = \cR(k-1, v)
$,
i.e. the \emph{$(k-1)$th-order Reed-Muller code}
of length $2^v$ and minimum distance $2^{v-k+1}$.
Also,
$
    C_\cG = \cR(k-1, v)^*\cap {\F_2\mbox{\j}}^\perp
$,
i.e. the subcode consisting of the even-weight codewords of the \emph{punctured
$(k-1)$th-order Reed-Muller code}
of length $2^v-1$ and minimum distance $2^{v-k+1}$.
% Minimum distances are in MacWilliams, Sloane page 375, 377
% Punctured RM code: $d = 2^{v-k+1}-1$
For $p = q = 2$ equation (\ref{eq:hamada}) simplifies to
$
	\rank_2 N_\cG = \sum_{i=0}^{v-k} \binom{v}{i}
$,
see \cite[p. 151]{assmus_key_1992}.

% % % % % % % % % % % % % % % % % % % % % % % % % % % % % % % % % % % % % % % % % % % % % % %
\section{Subspace designs}
Subspace designs---also called $q$-analogs of designs---were introduced
independently by Ray-Chaudhuri, Cameron, Delsarte in the early 1970s compare ~\cite{qdesigns2017}.

Let $q$ be a prime power $p^m$ and $V$ be a vector space of finite dimension $v$ over the finite field $\F_q$.

\begin{defini}\label{def:qdesign}
Let $0\leq t\leq k\leq v$ be integers and $\lambda$ a non-negative integer.
A pair $\cD=(V,\cB)$, where $\cB$ is a collection of $k$-subspaces (\emph{blocks}) of $V$, is called a
\emph{$t$-$(v,k,\lambda)_q$ subspace design} on $V$
if each $t$-subspace of $V$ is contained in exactly $\lambda$ blocks.

If $\cB$ is a set, i.e.~if every $k$-subspace appears at most once in $\cB$, the design is called \emph{simple}.
\end{defini}

The first nontrivial subspace design for $t\geq 2$ was constructed by Thomas (1987), the
first (and so far only known) nontrivial $t$-$(v, k, 1)_q$ subspace designs
(called $q$-Steiner systems) were constructed recently~\cite{fmp:10491987}.

In the rest of this note, all designs---combinatorial designs and subspace designs---will be simple
and we will omit mentioning this.
In order to distinguish subspace designs from those from Definition~\ref{def:design}, we will call the latter
\emph{combinatorial designs}.

A few well known facts---see \cite{qdesigns2017} for an overview---show the analogy between
combinatorial designs and subspace designs:
Let $\cD$ be a $t$-$(v, k, \lambda)_q$ design.
In general, for $0\le s \le t$, $\cD$ is also a
	$s$-$(v, k, \lambda_s)_q$ design, where
\[
\lambda_i = \lambda \frac{\qbinom{v-i}{t-i}{q}}{\qbinom{k-i}{t-i}{q}}\,.
\]
As a consequence, $\cD$ consists of
	$b = \lambda\qbinom{v}{t}{q} / \qbinom{k}{t}{q}$ blocks and
every $1$-dimensional
	subspace appears in $r =  \lambda\qbinom{v - 1}{t - 1}{q} / \qbinom{k- 1}{t-1}{q}$ blocks of $\cD$.

It is well known that $t$-$(v,k,\lambda)_q$ subspace designs gives rise to combinatorial
designs as is summarized in the following theorem.
\begin{thm}\label{thm:designparams}
Let $\cD=(V,\cB)$ be a $t$-$(v,k,\lambda)_q$ subspace design.
We define the following two combinatorial designs:
\begin{itemize}
	\item $\cD_p = (\qbinom{V}{1}{q}, \{\qbinom{B}{1}{q} \mid B\in\cB\})$ (projective version)
	\item For any fixed hyperplane $H\in\qbinom{V}{v-1}{q}$
	\[
		\cD_a = \left(\qbinom{V}{1}{q} \setminus \qbinom{H}{1}{q}, \left\{\qbinom{B}{1}{q} \setminus \qbinom{H}{1}{q} \mid B\in\cB,
        B \not \leq H\right\}\right)
	\]
    (affine version)
\end{itemize}
\begin{enumerate}
\item[a)] If $t = 2$, then $\cD_p$ is a combinatorial $2$-$(\qbinom{v}{1}{q},\qbinom{k}{1}{q},\lambda)$ design and
\item[b)] $\cD_a$ is a combinatorial $2$-$(q^{v-1},q^{k-1},\lambda)$ design.
\item[c)] In case $q=2$ and $t=3$, $\cD_a$ is even a combinatorial $3$-$(q^{v-1},q^{k-1},\lambda)$ design.
\item[d)]
In case $q=2$, a $2$-$(v,k,\lambda)_2$ design is also a combinatorial $3$-$(2^v, 2^k, \lambda)$ design.
\end{enumerate}
\end{thm}
\begin{proof}
For a)--c) see \cite{qdesigns2017}.

d) is a generalization of a result by Etzion and Vardy \cite{etzionVardy2011} for the case $\lambda=1$:

Let $V = \F_2^v$. The vector space $V$ consists of $2^v$ elements and every $k$-dimensional subspace $B$ contains
$2^k$ elements including $0$. Additionally, there exist $2^{v-k}$ flats parallel to $B$, denoted by $\cF_B$.
On the other hand, every tripel $\{a,b,c\}\subset V$ can be mapped to a unique two-dimensional subspace spanned by
$\ell := \langle b-a, c-a\rangle \leq V$. Then, $\ell$ is contained in exactly $\lambda$ blocks $b\in\cB$ and
$\{a,b,c\}$ is contained in exactly $\lambda$ flats in $\bigcup_{B\in\cB} \cF_B$.

This proves that $(V, \bigcup_{B\in\cB} \cF_B)$ is a combinatorial $3$-$(2^v, 2^k, \lambda)$ design.
\end{proof}

\begin{rem}
\begin{enumerate}
\item
Rahman and Blake \cite{RahmanBlake1973} analyzed the majority logic decoding capability
for combinatorial designs with $t>2$.

\item
As subspace design, the set of blocks of a geometric design with the above parameters
is the trivial
\[
	t\text{-}(v, k, \qbinom{v - t}{k - t}{q})_q
\]
subspace design for all $0\leq t \leq k$, compare~\cite{qdesigns2017}.
\end{enumerate}
\end{rem}

% % % % % % % % % % % % % % % % % % % % % % % % % % % % % % % % % % % % % % % % % % % % % % %
\section{One-step majority logic decoding with subspace designs}\label{Sec:onestep}
Let $\cD$ be a $t$-$(v, k, \lambda)_q$ subspace design.
Then, in the projective case $\cD$ can be regarded as combinatorial
$2\text{-}(\qbinom{v}{1}{q}, \qbinom{k}{1}{q}, \lambda_2)$
design. The rows of it's blocks-points-incidence matrix $N_\cD$ are a subset of the rows of the
incidence matrix $N_\cG$ of the
$2\text{-}(\qbinom{v}{1}{q}, \qbinom{k}{1}{q}, \qbinom{v - 2}{k - 2}{q})$ geometric design $\cG$.

Now, the simple observation is that if rows are removed from a matrix,
it's rank either stays constant or becomes smaller. Therefore:
\[
\rank_p N_\cD \le \rank_p N_\cG \,.
\]
So far, in all tested subspace designs for $q=p= 2$ we had
\[
\rank_2 N_\cD = \rank_2 N_\cG = \sum_{i=0}^{v-k} \binom{v}{i} \,.
\]

We conclude that codes $C_\cD$ from subspace designs are either the same codes as the codes $C_\cG$ from the
corresponding geometric designs or contain these codes.

What about the error correction capability of the one-step majority logic decoder?
The number of errors $\ell$ which can be corrected by
one-step majority logic decoding of a $2$-$(v,k,\lambda)$ design is
$
\ell = \lfloor \frac{r + \lambda - 1}{2\lambda} \rfloor
$.
Using $r = \lambda \binom{v - 1}{t - 1} / \binom{k - 1}{t - 1}$, $\ell$ can be bounded by
\[
\lfloor(\frac{q^{v-1}-1}{q^{k-1}-1} - 1) / 2\rfloor
\leq \ell \leq
\lfloor(\frac{q^{v-1}-1}{q^{k-1}-1} - 1) / 2 + \frac{1}{2\lambda}\rfloor\,.
\]
So, in fact, $\lambda$ cancels out and we see that the choice of $\lambda$ is irrelevant
for the error-correction capability of the code.

The advantage of taking a subspace design with small $\lambda$ over the trivial design
is in the reduced complexity of the decoder. For every coordinate
of a received word,
labeled by the $1$-dimensional subspaces of $V$ (the points),
the decoder runs through
those $r$ blocks of the design which contain that point.
Therefore, subspace designs with small values of $\lambda$ are preferable and
the trivial subspace design is clearly the worst choice
since it attains the maximal value of $\lambda$.
In case of $ \lambda=1$, the check equations are \emph{orthogonal}, i.e.
each coordinate pair appears in exactly one check equation.

The observation on the rank of a subspace design is also true for affine subspace designs
and for small values of $\lambda$ the resulting codes will have efficient decoders with
the same capabilities as those from the geometric designs.

The rapid growth of the number $r$ with increasing $v$ is the reason why for practical purposes among the
geometric designs, mostly the $2$-$(v, 2, 1)_q$ designs have been considered.
By using subspace designs the choice for suitable codes is much larger.

In  Tables~\ref{tab:projcodes2}--\ref{tab:q4codes} the parameters of the codes from known small
subspace designs are listed.
$\lambda_{\mbox{\tiny  known}}$ is the minimal value of $\lambda$ for which a subspace design
$\cD$ is known to exist,
$\lambda_{\min}$ is the minimal value of $\lambda$ that satisfies the necessary conditions
and $\lambda_{\max}$ is the value of $\lambda$ of the geometric design.
$r$ is the repetition number for $\lambda_{\mbox{\tiny  known}}$.
The next column contains the parameters of the resulting linear code $C_\cD$.
$n$ is the length of the code, dim is the dimension and $\ell$ is the number of errors
which can be corrected by one-step majority logic decoding according to
Rudolph, Ng~\cite{Rudolph1967,Ng1970} or Rahman, Blake~\cite{RahmanBlake1973}.

The column $r_{\max}/r$ shows the reduction factor for the number of parity check equations
if taking the subspace design with the smallest
known $\lambda$ against
taking the geometric design, i.e. it is the ratio between the entries of
$\lambda_{\max}$ and $\lambda_\emph{known}$.
This column gives the speed improvement for the decoder when using the best known subspace design.
If there is no entry in this column, no subspace design with smaller $\lambda$ is known or possible.
Finally, the last column indicates if a cyclic subspace design is known to exist.

\begin{rem}
The tables in \cite{qdesigns2017,qdesignscomputer2017} contain many subspace designs which are invariant
under a Singer cycle. For these designs, all positions can be decoded by the same decoder.
This reduces the complexity of the decoder by the factor $\qbinom{v}{1}{q}$.
\end{rem}

% % % % % % % % % % % % % % % % % % % % % % % % % % % % % % % % % % % % % % % % % % % % % % %
%\section{Affine subspace designs}
%Let $\cD=(V,\cB)$ be a $t$-$(v,k,\lambda)_q$ subspace design.
%For any fixed hyperplane $H\in\qbinom{V}{v-1}{q}$ let
%\[%
%		\cD_a = \left(
%			\qbinom{V}{1}{q} \setminus \qbinom{H}{1}{q},
%			\left\{\qbinom{B}{1}{q} \setminus \qbinom{H}{1}{q} \mid B\in\cB, B \not \leq H\right\}
%		\right)\,.
%\]
%$\cD_a$ is called \emph{affine subspace design}.
%If $t = 2$, then $\cD_a$ is a combinatorial $2$-$(q^{v-1},q^{k-1},\lambda)$ design.
%In the case $q=2$ and $t=3$, $\cD_a$ is even a combinatorial $3$-$(q^{v-1},q^{k-1},\lambda)$ design,
%see~\cite{qdesigns2017}.

%In case $q = p = 2$, the $2$-rank of the affine geometric design with parameters
%$2$-$(q^{v-1},q^{k-1}, \qbinom{v - 2}{k - 2}{q})$ is equal to
%\[
%	\rank_2 N_\cG = \sum_{i=0}^{v-k} \binom{v-1}{i},
%\]
%
%The observation on the rank of a subspace design is also true for affine subspace designs
%and for small values of $\lambda$ the resulting codes will have efficient decoders with the same capabilities
%as those from the geometric designs.

\begin{table}[!t]
\caption{Code parameters from the construction in Theorem \ref{thm:designparams} a)
for known subspace designs over $\F_2$.}\label{tab:projcodes2}
\centering
{\footnotesize\begin{tabular}{lrrlrrc}
\hline
$t$-$(v,k,\lambda_{\mbox{\tiny  known}})_2$  &
$\lambda_{\min}$ &
$\lambda_{\max}$ &
$[n, \mbox{dim}, \ell]_2$ &
$r$&
$r_{\max}/r$ &
cyclic \\
\hline\hline
$2$-$(3,2,1)_2$ & $1$ & $1$ & $[7, 3, 1]$ & $3$ & $$ & c\\
$2$-$(4,2,1)_2$ & $1$ & $1$ & $[15, 4, 3]$ & $7$ & $$ & c\\
$2$-$(4,3,3)_2$ & $3$ & $3$ & $[15, 10, 1]$ & $7$ & $$ & c\\
$2$-$(5,2,1)_2$ & $1$ & $1$ & $[31, 5, 7]$ & $15$ & $$ & c\\
$2$-$(5,3,7)_2$ & $7$ & $7$ & $[31, 15, 2]$ & $35$ & $$ & c\\
$2$-$(5,4,7)_2$ & $7$ & $7$ & $[31, 25, 1]$ & $15$ & $$ & c\\
$2$-$(6,2,1)_2$ & $1$ & $1$ & $[63, 6, 15]$ & $31$ & $$ & c\\
$2$-$(6,3,3)_2$ & $3$ & $15$ & $[63, 21, 5]$ & $31$ & $5.0$ & \\
$2$-$(6,4,35)_2$ & $35$ & $35$ & $[63, 41, 2]$ & $155$ & $$ & c\\
$2$-$(6,5,15)_2$ & $15$ & $15$ & $[63, 56, 1]$ & $31$ & $$ & c\\
$2$-$(7,2,1)_2$ & $1$ & $1$ & $[127, 7, 31]$ & $63$ & $$ & c\\
$2$-$(7,3,3)_2$ & $1$ & $31$ & $[127, 28, 10]$ & $63$ & $10.3$ & c\\
$2$-$(7,4,15)_2$ & $5$ & $155$ & $[127, 63, 4]$ & $135$ & $10.3$ & c\\
$2$-$(7,5,155)_2$ & $155$ & $155$ & $[127, 98, 2]$ & $651$ & $$ & c\\
$2$-$(7,6,31)_2$ & $31$ & $31$ & $[127, 119, 1]$ & $63$ & $$ & c\\
$2$-$(8,2,1)_2$ & $1$ & $1$ & $[255, 8, 63]$ & $127$ & $$ & c\\
$2$-$(8,3,21)_2$ & $21$ & $63$ & $[255, 36, 21]$ & $889$ & $3.0$ & c\\
$2$-$(8,4,7)_2$ & $7$ & $651$ & $[255, 92, 9]$ & $127$ & $93.0$ & \\
$2$-$(8,5,465)_2$ & $465$ & $1395$ & $[255, 162, 4]$ & $3937$ & $3.0$ & c\\
$2$-$(8,6,651)_2$ & $651$ & $651$ & $[255, 218, 2]$ & $2667$ & $$ & c\\
$2$-$(8,7,63)_2$ & $63$ & $63$ & $[255, 246, 1]$ & $127$ & $$ & c\\
$2$-$(9,2,1)_2$ & $1$ & $1$ & $[511, 9, 127]$ & $255$ & $$ & c\\
$2$-$(9,3,7)_2$ & $1$ & $127$ & $[511, 45, 42]$ & $595$ & $18.1$ & \\
$2$-$(9,4,21)_2$ & $7$ & $2667$ & $[511, 129, 18]$ & $765$ & $127.0$ & c\\
$2$-$(9,5,93)_2$ & $31$ & $11811$ & $[511, 255, 8]$ & $1581$ & $127.0$ & c\\
$2$-$(9,6,651)_2$ & $93$ & $11811$ & $[511, 381, 4]$ & $5355$ & $18.1$ & \\
$2$-$(9,8,127)_2$ & $127$ & $127$ & $[511, 501, 1]$ & $255$ & $$ & c\\
$2$-$(10,2,1)_2$ & $1$ & $1$ & $[1023, 10, 255]$ & $511$ & $$ & c\\
$2$-$(10,3,15)_2$ & $3$ & $255$ & $[1023, 55, 85]$ & $2555$ & $17.0$ & c\\
$2$-$(10,4,595)_2$ & $5$ & $10795$ & $[1023, 175, 36]$ & $43435$ & $18.1$ & \\
$2$-$(10,5,765)_2$ & $15$ & $97155$ & $[1023, 385, 17]$ & $26061$ & $127.0$ & \\
$2$-$(10,9,255)_2$ & $255$ & $255$ & $[1023, 1012, 1]$ & $511$ & $$ & c\\
$2$-$(11,2,1)_2$ & $1$ & $1$ & $[2047, 11, 511]$ & $1023$ & $$ & c\\
$2$-$(11,3,7)_2$ & $7$ & $511$ & $[2047, 66, 170]$ & $2387$ & $73.0$ & c\\
%$2$-$(11,5,43435)_2$ & $5$ & $788035$ & $[2047, 561, 34]$ & $2962267$ & $18.1$ & \\
$2$-$(11,10,511)_2$ & $511$ & $511$ & $[2047, 2035, 1]$ & $1023$ & $$ & c\\
$2$-$(12,2,1)_2$ & $1$ & $1$ & $[4095, 12, 1023]$ & $2047$ & $$ & c\\
$2$-$(12,3,1023)_2$ & $3$ & $1023$ & $[4095, 78, 341]$ & $698027$ & $$ & c\\
%$2$-$(12,6,2962267)_2$ & $31$ & $53743987$ & $[4095, 1585, 33]$ & $195605179$ & $18.1$ & \\
$2$-$(12,11,1023)_2$ & $1023$ & $1023$ & $[4095, 4082, 1]$ & $2047$ & $$ & c\\
$2$-$(13,2,1)_2$ & $1$ & $1$ & $[8191, 13, 2047]$ & $4095$ & $$ & c\\
$2$-$(13,3,1)_2$ & $1$ & $2047$ & $[8191, 91, 682]$ & $1365$ & $2047.0$ & c\\
$2$-$(13,12,2047)_2$ & $2047$ & $2047$ & $[8191, 8177, 1]$ & $4095$ & $$ & c\\
\hline
\end{tabular}}
\end{table}

\begin{table}[!t]
\caption{Code parameters from the construction in Theorem \ref{thm:designparams} b)
for known subspace designs over $\F_2$.}\label{tab:affinecodes}
\centering
{\small\begin{tabular}{lrrlrrc}
\hline
%$v$  &  $k$ & $\lambda_{\mbox{\tiny  known}}$ & $\lambda_{\min}$ & $\lambda_{\max}$ & $r$& $(n, \mbox{dim}, \ell)_2$ &  $r_{\max}/r$   \\ \hline
$t$-$(v,k,\lambda_{\mbox{\tiny  known}})_2$  &
$\lambda_{\min}$ &
$\lambda_{\max}$ &
$[n, \mbox{dim}, \ell]_2$ &
$r$&
$r_{\max}/r$ &
\\
\hline\hline
$2$-$(3,2,1)_2$ & $1$ & $1$ & $[4, 1, 1]$ & $3$ & $$ & \\
$2$-$(4,2,1)_2$ & $1$ & $1$ & $[8, 1, 3]$ & $7$ & $$ & \\
$2$-$(4,3,3)_2$ & $3$ & $3$ & $[8, 4, 1]$ & $7$ & $$ & \\
$2$-$(5,2,1)_2$ & $1$ & $1$ & $[16, 1, 7]$ & $15$ & $$ & \\
$2$-$(5,3,7)_2$ & $7$ & $7$ & $[16, 5, 2]$ & $35$ & $$ & \\
$2$-$(5,4,7)_2$ & $7$ & $7$ & $[16, 11, 1]$ & $15$ & $$ & \\
$2$-$(6,2,1)_2$ & $1$ & $1$ & $[32, 1, 15]$ & $31$ & $$ & \\
$2$-$(6,3,3)_2$ & $3$ & $15$ & $[32, 6, 5]$ & $31$ & $5.0$ & \\
$2$-$(6,4,35)_2$ & $35$ & $35$ & $[32, 16, 2]$ & $155$ & $$ & \\
$2$-$(6,5,15)_2$ & $15$ & $15$ & $[32, 26, 1]$ & $31$ & $$ & \\
$2$-$(7,2,1)_2$ & $1$ & $1$ & $[64, 1, 31]$ & $63$ & $$ & \\
$2$-$(7,3,3)_2$ & $1$ & $31$ & $[64, 7, 10]$ & $63$ & $10.3$ & \\
$2$-$(7,4,15)_2$ & $5$ & $155$ & $[64, 22, 4]$ & $135$ & $10.3$ & \\
$2$-$(7,5,155)_2$ & $155$ & $155$ & $[64, 42, 2]$ & $651$ & $$ & \\
$2$-$(7,6,31)_2$ & $31$ & $31$ & $[64, 57, 1]$ & $63$ & $$ & \\
$2$-$(8,2,1)_2$ & $1$ & $1$ & $[128, 1, 63]$ & $127$ & $$ & \\
$2$-$(8,3,21)_2$ & $21$ & $63$ & $[128, 8, 21]$ & $889$ & $3.0$ & \\
$2$-$(8,4,7)_2$ & $7$ & $651$ & $[128, 29, 9]$ & $127$ & $93.0$ & \\
$3$-$(8,4,11)_2$ & $1$ & $31$ & $[128, 29, 9]$ & $4191$ & $2.8$ & \\
$2$-$(8,5,465)_2$ & $465$ & $1395$ & $[128, 64, 4]$ & $3937$ & $3.0$ & \\
$2$-$(8,6,651)_2$ & $651$ & $651$ & $[128, 99, 2]$ & $2667$ & $$ & \\
$2$-$(8,7,63)_2$ & $63$ & $63$ & $[128, 120, 1]$ & $127$ & $$ & \\
$2$-$(9,2,1)_2$ & $1$ & $1$ & $[256, 1, 127]$ & $255$ & $$ & \\
$2$-$(9,3,7)_2$ & $1$ & $127$ & $[256, 9, 42]$ & $595$ & $18.1$ & \\
$2$-$(9,4,21)_2$ & $7$ & $2667$ & $[256, 37, 18]$ & $765$ & $127.0$ & \\
$2$-$(9,5,93)_2$ & $31$ & $11811$ & $[256, 93, 8]$ & $1581$ & $127.0$ & \\
$2$-$(9,6,651)_2$ & $93$ & $11811$ & $[256, 163, 4]$ & $5355$ & $18.1$ & \\
$2$-$(9,8,127)_2$ & $127$ & $127$ & $[256, 247, 1]$ & $255$ & $$ & \\
$2$-$(10,2,1)_2$ & $1$ & $1$ & $[512, 1, 255]$ & $511$ & $$ & \\
$2$-$(10,3,15)_2$ & $3$ & $255$ & $[512, 10, 85]$ & $2555$ & $17.0$ & \\
$2$-$(10,4,595)_2$ & $5$ & $10795$ & $[512, 46, 36]$ & $43435$ & $18.1$ & \\
$2$-$(10,5,765)_2$ & $15$ & $97155$ & $[512, 130, 17]$ & $26061$ & $127.0$ & \\
$2$-$(10,9,255)_2$ & $255$ & $255$ & $[512, 502, 1]$ & $511$ & $$ & \\
$2$-$(11,2,1)_2$ & $1$ & $1$ & $[1024, 1, 511]$ & $1023$ & $$ & \\
$2$-$(11,3,7)_2$ & $7$ & $511$ & $[1024, 11, 170]$ & $2387$ & $73.0$ & \\
%$2$-$(11,5,43435)_2$ & $5$ & $788035$ & $[1024, 176, 34]$ & $2962267$ & $18.1$ & \\
$2$-$(11,10,511)_2$ & $511$ & $511$ & $[1024, 1013, 1]$ & $1023$ & $$ & \\
$2$-$(12,2,1)_2$ & $1$ & $1$ & $[2048, 1, 1023]$ & $2047$ & $$ & \\
$2$-$(12,3,1023)_2$ & $3$ & $1023$ & $[2048, 12, 341]$ & $698027$ & $$ & \\
%$2$-$(12,6,2962267)_2$ & $31$ & $53743987$ & $[2048, 562, 33]$ & $195605179$ & $18.1$ & \\
$2$-$(12,11,1023)_2$ & $1023$ & $1023$ & $[2048, 2036, 1]$ & $2047$ & $$ & \\
$2$-$(13,2,1)_2$ & $1$ & $1$ & $[4096, 1, 2047]$ & $4095$ & $$ & \\
$2$-$(13,3,1)_2$ & $1$ & $2047$ & $[4096, 13, 682]$ & $1365$ & $2047.0$ & \\
$2$-$(13,12,2047)_2$ & $2047$ & $2047$ & $[4096, 4083, 1]$ & $4095$ & $$ & \\
\hline
\end{tabular}}
\end{table}

\begin{table}[!t]
\caption{Code parameters from the construction in Theorem \ref{thm:designparams} d)
for known subspace designs over $\F_2$.}\label{tab:affinecodesII}
\centering
{\small\begin{tabular}{lrrlrrc}
\hline
$t$-$(v,k,\lambda_{\mbox{\tiny  known}})_2$  &
$\lambda_{\min}$ &
$\lambda_{\max}$ &
$[n, \mbox{dim}, \ell]_2$ &
$r$&
$r_{\max}/r$ &
\\
\hline\hline
$2$-$(3,2,1)_2$ & $1$ & $1$ & $[8, 4, 1]$ & $7$ & $$ & \\
$2$-$(4,2,1)_2$ & $1$ & $1$ & $[16, 5, 3]$ & $35$ & $$ & \\
$2$-$(4,3,3)_2$ & $3$ & $3$ & $[16, 11, 1]$ & $15$ & $$ & \\
$2$-$(5,2,1)_2$ & $1$ & $1$ & $[32, 6, 5]$ & $155$ & $$ & \\
$2$-$(5,3,7)_2$ & $7$ & $7$ & $[32, 16, 3]$ & $155$ & $$ & \\
$2$-$(5,4,7)_2$ & $7$ & $7$ & $[32, 26, 1]$ & $31$ & $$ & \\
$2$-$(6,2,1)_2$ & $1$ & $1$ & $[64, 7, 11]$ & $651$ & $$ & \\
$2$-$(6,3,3)_2$ & $3$ & $15$ & $[64, 22, 5]$ & $279$ & $5.0$ & \\
$2$-$(6,4,35)_2$ & $35$ & $35$ & $[64, 42, 2]$ & $651$ & $$ & \\
$2$-$(6,5,15)_2$ & $15$ & $15$ & $[64, 57, 1]$ & $63$ & $$ & \\
$2$-$(7,2,1)_2$ & $1$ & $1$ & $[128, 8, 21]$ & $2667$ & $$ & \\
$2$-$(7,3,3)_2$ & $1$ & $31$ & $[128, 29, 9]$ & $1143$ & $10.3$ & \\
$2$-$(7,4,15)_2$ & $5$ & $155$ & $[128, 64, 5]$ & $1143$ & $10.3$ & \\
$2$-$(7,5,155)_2$ & $155$ & $155$ & $[128, 99, 2]$ & $2667$ & $$ & \\
$2$-$(7,6,31)_2$ & $31$ & $31$ & $[128, 120, 1]$ & $127$ & $$ & \\
$2$-$(8,2,1)_2$ & $1$ & $1$ & $[256, 9, 43]$ & $10795$ & $$ & \\
$2$-$(8,3,21)_2$ & $21$ & $63$ & $[256, 37, 19]$ & $32385$ & $3.0$ & \\
$2$-$(8,4,7)_2$ & $7$ & $651$ & $[256, 93, 9]$ & $2159$ & $93.0$ & \\
$2$-$(8,5,465)_2$ & $465$ & $1395$ & $[256, 163, 5]$ & $32385$ & $3.0$ & \\
$2$-$(8,6,651)_2$ & $651$ & $651$ & $[256, 219, 2]$ & $10795$ & $$ & \\
$2$-$(8,7,63)_2$ & $63$ & $63$ & $[256, 247, 1]$ & $255$ & $$ & \\
$2$-$(9,2,1)_2$ & $1$ & $1$ & $[512, 10, 85]$ & $43435$ & $$ & \\
$2$-$(9,3,7)_2$ & $1$ & $127$ & $[512, 46, 37]$ & $43435$ & $18.1$ & \\
$2$-$(9,4,21)_2$ & $7$ & $2667$ & $[512, 130, 17]$ & $26061$ & $127.0$ & \\
$2$-$(9,5,93)_2$ & $31$ & $11811$ & $[512, 256, 9]$ & $26061$ & $127.0$ & \\
$2$-$(9,6,651)_2$ & $93$ & $11811$ & $[512, 382, 4]$ & $43435$ & $18.1$ & \\
$2$-$(9,8,127)_2$ & $127$ & $127$ & $[512, 502, 1]$ & $511$ & $$ & \\
$2$-$(10,2,1)_2$ & $1$ & $1$ & $[1024, 11, 171]$ & $174251$ & $$ & \\
$2$-$(10,3,15)_2$ & $3$ & $255$ & $[1024, 56, 73]$ & $373395$ & $17.0$ & \\
$2$-$(10,4,595)_2$ & $5$ & $10795$ & $[1024, 176, 35]$ & $2962267$ & $18.1$ & \\
$2$-$(10,5,765)_2$ & $15$ & $97155$ & $[1024, 386, 17]$ & $860013$ & $127.0$ & \\
$2$-$(10,9,255)_2$ & $255$ & $255$ & $[1024, 1013, 1]$ & $1023$ & $$ & \\
$2$-$(11,2,1)_2$ & $1$ & $1$ & $[2048, 12, 341]$ & $698027$ & $$ & \\
$2$-$(11,3,7)_2$ & $7$ & $511$ & $[2048, 67, 147]$ & $698027$ & $73.0$ & \\
%$2$-$(11,5,43435)_2$ & $5$ & $788035$ & $[2048, 562, 33]$ & $195605179$ & $18.1$ & \\
$2$-$(11,10,511)_2$ & $511$ & $511$ & $[2048, 2036, 1]$ & $2047$ & $$ & \\
$2$-$(12,2,1)_2$ & $1$ & $1$ & $[4096, 13, 683]$ & $2794155$ & $$ & \\
$2$-$(12,3,1023)_2$ & $3$ & $1023$ & $[4096, 79, 293]$ & $408345795$ & $$ & \\
%$2$-$(12,6,2962267)_2$ & $31$ & $53743987$ & $[4096, 1586, 33]$ & $12714336635$ & $18.1$ & \\
$2$-$(12,11,1023)_2$ & $1023$ & $1023$ & $[4096, 4083, 1]$ & $4095$ & $$ & \\
$2$-$(13,2,1)_2$ & $1$ & $1$ & $[8192, 14, 1365]$ & $11180715$ & $$ & \\
$2$-$(13,3,1)_2$ & $1$ & $2047$ & $[8192, 92, 585]$ & $1597245$ & $2047.0$ & \\
$2$-$(13,12,2047)_2$ & $2047$ & $2047$ & $[8192, 8178, 1]$ & $8191$ & $$ & \\
\hline
\end{tabular}}
\end{table}

\begin{table}[!t]
\caption{Code parameters from the construction in Theorem \ref{thm:designparams} a)
for known subspace designs over $\F_4$.}\label{tab:q4codes}
\centering
{\small\begin{tabular}{lrrrlr}
\hline
%$v$  &  $k$ & $\lambda_{\mbox{\tiny  known}}$ & $\lambda_{\min}$ & $\lambda_{\max}$ & $r$& $(n, \mbox{dim}, \ell)_2$ &  $r_{\max}/r$  \\ \hline
$t$-$(v,k,\lambda_{\mbox{\tiny  known}})_4$  &
$\lambda_{\min}$ & $\lambda_{\max}$ & $r$& $[n, \mbox{dim}, \ell]_2$ &  $r_{\max}/r$  \\
\hline\hline
$2$-$(7,3,21)_4$ & 1 & 341 & 5733 & [5461, 1064, 136] & 16.2 \\
$2$-$(7,4,357)_4$ & 17 & 5797 & 23205 & [5461, 3185, 32] & 16.2 \\
\hline
\end{tabular}}
\end{table}

% % % % % % % % % % % % % % % % % % % % % % % % % % % % % % % % % % % % % % % % % % % % % % %

\section{Two-step majority logic decoding with subspace designs}

It is well known that for many codes one-step majority logic decoding can decode
much less errors than one-half of the minimum distance,
see e.g. \cite[Ch.~10]{PetersonWeldon1972}.
In 1954, Reed \cite{Reed1954} developed a multi-step majority logic decoding algorithm
which corrects up to one-half of the minimum distance errors for certain codes,
among them the Reed-Muller codes and the geometric codes.
Alternatively, Peterson and Weldon proposed in \cite[p.~336]{PetersonWeldon1972}
a two-step majority logic decoding of
geometric codes that can correct %more errors than the one-step variant.
also up to one-half of the minimum distance errors for geometric codes.
Here, we adapt the latter method to majority logic decoding based on subspace designs.

To recall the approach by Peterson and Weldon, consider the projective geometry code $C_{\cG}$
obtained from the $2$-$(\qbinom{v}{1}{q}, \qbinom{k}{1}{q}, \qbinom{v-2}{k-2}{q})$
geometric design $\cG$. The same approach will work also for Euclidean codes.

%Let $N_{\cG}$ be the incidence matrix of the design whose rows span the dual code $C^{\bot}_{\cG}$.
As in Section~\ref{Sec:onestep}, the decoding process uses a $t$-$(v,k-1,\lambda)_q$ subspace design
$\cD = (V, \cB)$.
Recall that $\cD$ is also a combinatorial $2$-$(\qbinom{v}{1}{q}, \qbinom{k-1}{1}{q}, \lambda)$ design.
The decoding process is done in two steps:
In the first step we will consider for each block $B\in\cB$ the set
$K_B = \{ K\in \qbinom{V}{k}{q} \mid B \leq K\}$ of all $k$-dimensional subspaces containing $B$ and
in the second step we use the $(k-1)$-dimensional blocks $\cB$.

Assuming a word $w = c + e$ has been received, the decoding algorithms now works as follows:
\begin{enumerate}
\item For each block $B\in\cB$:
	\begin{enumerate}
	\item For each $K \in K_B$ determine an estimate of $\sum_{i\in B} w_i$ by evaluating the
	parity check equations
	\[
		(\sum_{i\in B} w_i)_K := - \sum_{j \in K\setminus B} w_j
	\]
	\item By majority decision over all $(\sum_{i\in B} w_i)_K, K\in K_B$,
		determine the correct value of $\sum_{i\in B} w_i$.
	\end{enumerate}
\item For each point $j$, determine by majority decision over all $\sum_{i\in B} w_i$,
	$B \in \{ B \in \cB \mid j \in B\}$, the correct value of $w_j$.
\end{enumerate}

Note that the first step is part of the usual $L$-step majority logic decoding
using orthogonal check equations (or the Reed algorithm).
It is also used in the first step of the algorithm by Peterson and Weldon.
However, while Peterson and Weldon had to evaluate all $k$-dimensional subspaces in Step 1,
because they used all $(k-1)$-dimensional subspaces in Step 2,
we only have to evalute all $k$-dimensional subspaces from
$\bigcup_{B\in\cB} K_B$.
So, in general the first step requires less computations than those of the Peterson-Weldon approach.
Ideally, we choose the subspace design $\cD$ with the least number of blocks
(follows from minimum $\lambda$).
The second step is the same as the one-step majority logic decoding with subspace designs
which in general has less complexity than using the geometric designs.

Now, we demonstrate the correctness of the proposed algorithm.
Recall that a $(k-1)$-subspace of $V$ is contained in exactly
\[ J=\qbinom{v-k+1}{1}{q} \]
$k$-subspaces.
These $J$ subspaces form a set of check equations orthogonal on the $(k-1)$-subspace.
Hence, in the first step, if there are at most
$\left\lfloor J/2\right\rfloor$ errors then a $(k-1)$-subspace can be correctly determined
from the $k$-subspaces containing it using majority logic decoding.

For the second step, we know that it can correct at most
$\left\lfloor (r+\lambda -1)/2\lambda \right\rfloor$ errors
where $r$ is the number of blocks of $\mathcal{D}$ containing a given point.

The two-step algorithm will determine the correct codeword if the majority logic decisions
in both steps are correct.
It follows that the two-step algorithm can correct at most
\[
	\min\{\left\lfloor J/2\right\rfloor,
	\left\lfloor (r+\lambda -1)/2\lambda\right\rfloor\}
\] errors.
We will prove that the minimum is $\left\lfloor J/2\right\rfloor$.
In
\cite{PetersonWeldon1972} it is shown that $\lfloor J/2\rfloor$ to be the same number of errors
that can be corrected using the Reed algorithm.

\begin{thm}%[Peterson and Weldon \cite{PetersonWeldon1972}]
We have $\left\lfloor (r+\lambda -1)/2\lambda\right\rfloor \leq \left\lfloor J/2\right\rfloor$.
\end{thm}
\begin{proof}
%It is sufficient to prove that $J\leq r/\lambda$ for the case of the
%%$2$-\left({v \brack 1}_q,{k-1 \brack 1}_q,{v-2 \brack k-3}_q\right)$
%$2$-$(\qbinom{v}{1}{q}, \qbinom{k}{1}{q}, \qbinom{v-2}{k-2}{q})$
%geometric design.
We have
\[ J={v-k+1 \brack 1}_q=\dfrac{q^{v-k+1}-1}{q-1}
\quad\mbox{and}\quad
r/\lambda= \dfrac{\lambda \qbinom{v-1}{1}{q}/\qbinom{k-2}{1}{q} }{\lambda}
 = \dfrac{q^{v-1}-1}{q^{k-2}-1}.
\]
The result follows from the fact that the inequality
\[
	\dfrac{q^{v-k+1}-1}{q-1}\leq \dfrac{q^{v-1}-1}{q^{k-2}-1}
\]
is equivalent to
\[ 2q^{v-1}+q\leq q^v+q^{v-k+1}+q^{k-2}.\]
\end{proof}
It is well known, see e.g. \cite{PetersonWeldon1972}, that the BCH bound of $C_\cG$
is $d_{\rm BCH} = \qbinom{v-k+1}{1}{q} + 1$.
Above we showed that the number of errors that can be corrected for $C_\cG$ and $C_\cD$ by two-step majority logic decoding is
$(d_{\rm BCH} -1)/2$.

According to \cite{LavrauwStormeVanDeVoorde} the true minimum distance of the
code $C_\cG$ is known only for the following cases:
\begin{itemize}
\item if $q=p=2$ then $d=2^{n-k+1}$
\item if $q>2$ is even then $d=(q+2)q^{v-k-1}$
\item if $k=v-1$ then $d=d_{\rm BCH}$
\end{itemize}
It follows that if $q=p=2$ or $k=v-1$ then $d= d_{\rm BCH}$.  In general, $d >  d_{\rm BCH}$.

% % % % % % % % % % % % % % % % % % % % % % % % % % % % % % % % % % % % % % % % % % % % % % %
\section{Future work}
For the codes from subspace designs an open question is if
we always have  $\rank_p N_\cD = \rank_p N_\cG$?
Further, in light of the above the Hamada conjecture has to be formulated more general.

\smallskip
\noindent
\textbf{Generalized Hamada conjecture.}
	Let $q$ be a power of a prime $p$ and
	let there be a $t$-$(v,k,\lambda)_q$ subspace design.
	Regarded as combinatorial design $\cD$,
	it has parameters  $2$-$(\qbinom{v}{1}{q}, \qbinom{k}{1}{q}, \lambda)$.
	The $p$-rank of $\cD$ is minimal among all
	combinatorial designs with the same parameters.

%\bibliographystyle{alpha}
%\bibliography{p_rank}
\bibliographystyle{IEEEtran}
\bibliography{IEEEabrv,p_rank}

%\newcommand{\etalchar}[1]{$^{#1}$}
%\begin{thebibliography}{BKW18b}
%\end{thebibliography}

\end{document}